\documentclass[11pt]{amsart}

\usepackage[utf8x]{inputenc} 
\usepackage[a4paper]{geometry} 
\usepackage{enumerate}  
\usepackage{amsmath} 
\usepackage{amssymb}
\usepackage{color}
\usepackage{multirow}
\usepackage{hyperref}  
\usepackage{verbatim}
\usepackage[normalem]{ulem}
\usepackage{comment}

                 \hypersetup{ pdfborder={0 0 0}, 
                              colorlinks=true, 
                              linktoc=page,
                              pdfauthor={Michael Hoff}, 
                              pdftitle={A note on syzygies and normal generation for trigonal curves}} 
\definecolor{Gray}{gray}{0.9}   

\usepackage[arrow,curve,matrix,tips,frame]{xy}

\usepackage{bbm}

\usepackage{mathrsfs}

\theoremstyle{definition}
\newtheorem{definition}{Definition}[section]
\newtheorem{example}[definition]{Example}
\newtheorem{remark}[definition]{Remark}

\theoremstyle{plain}
\newtheorem{theorem}[definition]{Theorem}
\newtheorem*{theorem*}{Theorem}

\newtheorem{lemma}[definition]{Lemma}

\theoremstyle{remark}

\newcommand{\OO}{\mathscr{O}}

\newcommand{\PP}{\mathbb{P}}

\newcommand{\Supp}{\text{\upshape{Supp}}}

\newcommand{\Pic}{\text{\upshape{Pic}}}

\newcommand{\cE}{\mathscr{E}}
\newcommand{\PPr}{{\mathbb{P}^r}}
\newcommand{\PPE}{{\PP(\cE)}}
\newcommand{\cC}{\mathscr{C}}
\newcommand{\cI}{\mathscr{I}}

\newcommand{\NN}{\mathbb{N}}

\numberwithin{equation}{section}

\title[A note on syzygies and normal generation for trigonal curves]{A note on syzygies and normal generation for trigonal curves}

\author[M. Hoff]{Michael Hoff} 
\address{Universit\"at des Saarlandes, Campus E2 4, D-66123 Saarbr\"ucken, Germany}
\email{\href{mailto:hahn@math.uni-sb.de}{hahn@math.uni-sb.de}} 

\begin{document}

\title{A note on syzygies and normal generation for trigonal curves}
\date{\today}

\thanks{This project was partially funded by the Deutsche Forschungsgemeinschaft (DFG, German Research Foundation) – Project-ID 286237555 – TRR 195.}

\begin{abstract}
Let $C$ be a trigonal curve of genus $g\ge 5$ and let $T$ be the unique trigonal line bundle inducing a map $\pi: C \stackrel{3:1}{\longrightarrow} \PP^1$. 
This note provides a short and easy proof of the normal generation for the residual line bundle $K_C\otimes T^{-1}$ for curves of genus $g\ge 7$. Moreover, we compute the minimal free resolution of the embedded curve $C\subset \PP(H^0(C,K_C\otimes T^{-n})^*)$ for the residual line bundle $K_C\otimes T^{-n}$ for $n\ge 1$ and $g \ge 3n+4$. 
\end{abstract}

\maketitle

\section{Introduction} \label{introduction}

In this short note we show how to determine the minimal free resolution of a trigonal curve lying on a rational normal surface scroll following \cite{Sch}. This was studied in a general context (arithmetic Cohen-Macaulay and non-arithmetic Cohen-Macaulay divisors on rational normal scrolls) by \cite{Na99} or \cite{Park14}, but we could not find any application for trigonal curves\footnote{\noindent A trigonal curve is a non-hyperelliptic smooth curve $C$ of genus $g$ with a line bundle $T$ inducing a $3:1$ morphism to $\PP^1$. For $g\ge 5$, the line bundle $T$ is unique.}. In particular, for a trigonal curve $C$ of genus $g\ge 5$ with trigonal bundle $T$ the shape of the minimal free resolution implies normal generation for line bundles of the form $K_C\otimes T^{-n}$ (by our knowledge,  the latter seems to be unknown for $n=1$ and small genus or for $n\ge 2$).  We present a short and selfcontained proof (see also Remark \ref{CastelnuovoVariety}). 

We summarize known result about normal generation of line bundles on trigonal curves. Let $C$ be a curve of genus $g$. A line bundle $L$ on $C$ is said to be normally generated if $L$ is very ample and the embedded curve $C\subset \PP(H^0(C,L)^*)$ is projectively normal (equivalently, the natural maps $S^m H^0(C,L) \to H^0(C,L^m)$ are surjective for all $m\ge 0$).
In their seminal paper \cite{GL86}, Green and Lazarsfeld prove that a very ample line bundle $L$ of degree $2g-k$ on $C$ is normally generated as long as $k$ is smaller than the Clifford index $c$ of $C$ \footnote{For $g\ge 4$ the Clifford index is defined as $\min\{\deg L - 2\cdot h^0(C,L) -2\ :\ L\in \Pic(C), h^i(C,L)\ge 2, i=0,1\}$.}. For our purpose the following result is important. 
\begin{theorem*}{\cite[Theorem 2]{GL86}}
 There exists an explicit constant $N(c)$ such that if $C$ is a curve of Clifford index $c$ and of genus $g > N(c)$, then every very ample line bundle $L$ with $h^1(C,L)\ge 2$ of degree $\deg (L) = 2g - 2\cdot h^1(C,L) - c$ is normally generated. 
\end{theorem*}
They furthermore showed the existence of non-special line bundles which are not normally generated, and classified not normally generated line bundles with $h^1(C,L)=1$. 

\begin{example}
 For a trigonal curve $C$ ($\Rightarrow$ Clifford index $1$) let $T$ be the trigonal bundle. We set $L=K_C\otimes T^{-1}$ in the above theorem. Then $K_C\otimes T^{-1}$ is normally generated if $g > N(1) = 16$. We could not find any further result in the literature concerning the normal generation of the line bundle $K_C\otimes T^{-1}$ for trigonal curves of genus $16\ge g\ge 7$ (the low genus cases are trivial by a B\'ezout argument). We close the remaining gap in the genus with our theorem and furthermore describe the whole minimal free resolution of $C\subset \PP(H^0(C,K_C\otimes T^{-1})^*)$.
\end{example}

Given a projective variety $Y\subset \PP^n$, we can resolve the structure sheaf of $Y$ by free $\OO_{\PP^n}$-modules, that is, a minimal free resolution 
$$
\OO_Y \leftarrow F_0 \leftarrow F_1\leftarrow \dots \leftarrow F_m \leftarrow 0,
$$
where $F_i = \bigoplus_j \OO_{\PP^n}(-j)^{\beta_{ij}}$. We call the numbers $\beta_{ij}$ the \emph{Betti numbers} of $Y$. The shape of the minimal free resolution will be encoded in the so-called \emph{Betti table} $(\beta_{ij})_{ij}$ (see \cite{Eis} for an introduction). 

Our main theorem is the following (see also Theorem \ref{genMainThm} for its generalization). 

\begin{theorem}\label{mainThm}
 Let $C$ be a trigonal curve of genus $g\ge 7$ and let $T$ be the unique trigonal bundle inducing the map $\pi: C \stackrel{3:1}{\longrightarrow} \PP^1$. The Betti table of a minimal free resolution of $C\subset \PP(H^0(C,K_C\otimes T^{-1})^*)=\PP^{g-3}$ has the following shape 
$$
\begin{tabular}{c|cccccccc}
 \multicolumn{1}{@{}}{\ }&& \\[-2mm] 
  & $0$ & $1$ & $\dots$ &  $g-6$ & $g-5$ & $g-4$  \\ \hline 
$0$ & 1 & 0 & $\cdots$ & 0 & 0 & 0 \\
$1$ & 0 & * & $\cdots$ & * & * & 0\\
$2$ & 0 & * & $\cdots$ & * & 0 & 0 \\
$3$ & 0 & 0 & $\cdots$ & 0 & * & *
 \\[-2mm]
 \multicolumn{1}{@{}}{\ }&&
\end{tabular}
$$
 where $*$ indicates nonzero entries. In particular, the embedded curve is projectively normal. 
\end{theorem}

\begin{remark}\label{CastelnuovoVariety}
Given a trigonal curve $C$ of genus $g$ and a positive number $n$, we provide a sharp range for $g$ and $n$ where $K_C\otimes T^{-n}$ is very ample in Section \ref{veryAmpleness}. Hence the embedded curve  $C\subset \PP(H^0(C,K_C\otimes T^{-n})^*)$ is smooth. One can show that projective normality of $K_C\otimes T^{-n}$ follows already from the classical Castelnuovo lemma \cite{Cas1889}. We note that trigonal curves $C\subset \PP(H^0(C,K_C\otimes T^{-n})^*)$ are extremal in the sense of Castelnuovo's bound and refer to \cite[Chapter III, \S 2, pp  120]{ACGH} for a detailed analysis.  Furthermore the statement of Theorem \ref{mainThm} also follows from \cite[Theorem 2.4 and Corollary 2.5]{Na99} using the smoothness assumption. We recall that we present a new selfcontained proof where our method of the proof is exactly valid in the range for $g$ and $n$ where $K_C\otimes T^{-n}$ is very ample. 
\end{remark}

In \cite{MS86}, Martens and Schreyer classified non-special not normally generated line bundles on trigonal curves. They used the extrinsic geometry of the embedding $C\subset \PP(H^0(C,L)^*)$ depending on the Maroni-invariant (see Section \ref{assScrollToT} for the definition). 

\begin{remark}
 In \cite{LN13}, Lange and Newstead determined all vector bundles on trigonal curves which compute all higher Clifford indices (see for the definition e.g., \cite{LN10}). Their main theorem \cite[Theorem 4.7]{LN13} relies on the normal generation of $K_C\otimes T^{-1}$ and is thus true for all trigonal curves of genus $g\ge 5$ by our main theorem. 
\end{remark}

\begin{remark}
 The nonzero Betti numbers in Theorem \ref{mainThm} are: 
 \begin{align*}
  & \beta_{j,j+1} = j \cdot \binom{g-4}{j+1}, &  \text{for }& 1 \le j \le g - 5\ \  \text{(first row)} \cr
  & \beta_{j,j+2} = (g - 5 - j) \cdot \binom{g-4}{j-1}, &  \text{for }& 1 \le j \le g - 6\ \ \text{(second row)} \cr
  & \beta_{j,j+3} = (j - g + 6) \cdot \binom{g-4}{j}, & \text{for } & g-5 \le j \le g-4\ \ \text{(third row)}
 \end{align*}
 These are exactly the numbers as in \cite[Theorem 2.4]{Na99}
\end{remark}

In Section \ref{proofMainThm} we present the proof of Theorem \ref{mainThm} and introduce the necessary background. In Section \ref{higherResiduals} we give a bound for the very ampleness of $K_C\otimes T^{-n}$ in terms of $n$ and the genus of the trigonal curve $C$. At the end we proof the generalization of Theorem \ref{mainThm}.

\section{Proof of the main theorem}\label{proofMainThm}
Our proof follows the strategy of \cite{Sch}: 
\begin{itemize}
 \item $C\subset \PP(H^0(C,K_C\otimes T^{-1})^*)$ lies on a rational normal surface $X$ swept out by $T$, 
 \item resolve $\OO_C$ as an $\OO_X$-module by direct sums of line bundles on $X$ and take the minimal free resolution of each of these line bundles as $\OO_{\PP^{g-3}}$-modules,
 \item a mapping cone construction induces a minimal free resolution of $\OO_C$ as an $\OO_{\PP^{g-3}}$-module.
\end{itemize}

\subsection{A rational normal surface associated to the trigonal bundle}\label{assScrollToT}

Let $C$ be a trigonal curve of genus $g\ge 5$. An important invariant associated to any trigonal curve - classically used to describe its Brill--Noether locus - is the so-called \emph{Maroni-invariant}. It is defined as follows. Let $\pi:C\to \PP^1$ be the map associated to the trigonal bundle $T$. The direct image $\pi_* K_C$ of the canonical bundle on $C$ splits into a direct sum of line bundles
$$
\pi_* K_C = \OO_{\PP^1}(a) \oplus \OO_{\PP^1}(m) \oplus \OO_{\PP^1}(-2) 
$$
where $m$ is the \emph{Maroni-invariant}. It is known that  
\begin{align}\label{MaroniInequalities}
 0< \frac{g-4}{3}\le m \le \frac{g-2}{2}. 
\end{align}
and $a = g-2-m$. Furthermore, $m+2$ and $a+2$ describe the two jumps of the function $f:\NN \to \{0,1,2\}$ with $f(n):=h^1(C,T^{n-1}) - h^1(C, T^n)$. Hence, 
\begin{align}\label{MaroniAsMin}
 m+2 = \text{min}\{n\in \NN\ :\ h^0(C,T^n) > n+1\}.
\end{align}
Maroni \cite{Mar46} introduced $m$ as a geometrical invariant determining a smooth rational normal scroll which is swept out by the trigonal bundle in the canonical space $\PP(H^0(C,K_C)^*)$. 

We will explain this procedure for our purpose in the following setting. We assume that $g\ge 6$. 
Let $C\subset \PP^{g-3} :=  \PP(H^0(C,K_C\otimes T^{-1})^*)$ be embedded by the very ample linear series $|K_C\otimes T^{-1}|$ (see Section \ref{veryAmpleness} and note that for $g=5$, $K_C\otimes T^{-1}$ maps $C$ to a plane quintic with one node). We consider the scroll swept out by the unique pencil $|T|$, that is,  
$$
X= \bigcup_{D\in |T|} \overline{D} \subset \PP^{g-3}
$$
where $\overline{D}$ is the linear span of the divisor $D$ defined by the linear forms 
$$H^0(C, K_C\otimes T^{-1}(-D)) \to H^0(C,K_C\otimes T^{-1}).$$ 
Since $H^0(C,K_C\otimes T^{-2})$ is $(g-4)$-dimensional (by (\ref{MaroniInequalities}) and (\ref{MaroniAsMin}) for $g\ge 6$), the linear span $\overline{D}$ is a line and $X$ is a rational normal surface of degree $g-4$. In particular, $X$ is the image of a projective bundle $\PP(\OO_{\PP^1}(e_1)\oplus \OO_{\PP^1}(e_2))$ where $e_1 = g - 3 - m$ and $e_2 = m - 1\ge 0$. We will explain this fact in the next section. 

\subsection{Resolution of the curve on the rational normal surface}\label{resOnScroll}

Let $C$ be a trigonal curve with Maroni-invariant $m$. Let $X$ be the rational normal surface of degree $g-4$ as above and let $\pi:\cE = \OO_{\PP^1}(e_1)\oplus \OO_{\PP^1}(e_2)\to \PP^1$ be the rank two bundle on $\PP^1$ with $e_1 = g - 3 - m$ and $e_2 = m - 1$. By \cite{H}, if $e_1,e_2>0$, then 
$$j:\PP(\cE)\to X\subset \PP H^0(\PPE,\OO_\PPE(1))= \PP^{g-3}$$ 
is an isomorphism. Otherwise it is a resolution of singularities. Since $R^i j_* \OO_\PPE=0$, it is convenient to consider $\PPE$ instead of $X$ for cohomological considerations. Note that this may happen for $g=6, 7$ and $m=1$. 
 \\
It is furthermore known that the Picard group $\Pic(\PPE)$ is generated 
by the ruling $R=[\pi^*\OO_{\PP^1}(1)]$ and the hyperplane class
$H=[j^*\OO_{\PP^{g-3}}(1)]$ with intersection products
$$
H^2=e_1 + e_2 = g-4, \ \ H\cdot R=1, \ \ R^2=0.
$$
Since $C\subset X$ and $T$ is base point free, we consider $C\subset \PPE$ as a subvariety of the projectivised bundle. Note that 
$$\OO_\PPE(H)\otimes \OO_C = K_C\otimes T^{-1} \text{ and } \OO_\PPE(R)\otimes \OO_C = T.$$ 

Since the curve $C$ is a codimension one subvariety of $\PPE$, the ideal sheaf $\cI_{C/\PPE}$ is generated by an element in $H^0(\PPE, \OO_\PPE(aH-bR))$. Since the ruling on $C$ has degree $3$ ($\Rightarrow$ $a=3$) and the degree of $C\subset \PP^{g-3}$ is $\deg(C)=2g-5$ ($\Rightarrow$ $b=g-7$), the resolution of $\OO_C$ as an $\OO_\PPE$-module is
\begin{align}\label{relRes}
 \begin{xy}
  \xymatrix{
  0 & \ar[l] \OO_C & \ar[l] \OO_\PPE & \ar[l] \OO_\PPE(-3H + (g-7)R) & \ar[l] 0 \\
  }
 \end{xy}
\end{align} 

In the next step we will resolve the line bundles in the above resolution in terms of $\OO_{\PP^{g-3}}$-modules. Therefore we recall the definition of \emph{Eagon--Northcott} type resolutions whereby we restrict to our case. 
 \\
We have a natural multiplication map
$$H^0(\PPE,\OO_\PPE(R))\otimes H^0(\PPE,\OO_\PPE(H-R))\longrightarrow H^0(\PPE,\OO_\PPE(H)),$$
and the equations of the rational normal surface $X$ are given by the $2\times 2$-minors of the matrix $\Phi$. We define 
$$
F:= H^0(\PPE,\OO_\PPE(H-R))\otimes \OO_{\PP^{g-3}}= \OO_{\PP^{g-3}}^{g-4},\ 
G:=H^0(\PPE,\OO_\PPE(R))\otimes \OO_{\PP^{g-3}}= \OO_{\PP^{g-3}}^2
$$
and regard $\Phi$ as a map $\Phi: F(-1)\to G$ (where we identify $G=G^*$). 

For $b\ge -1$ we can resolve $\OO_\PPE(aH+bR)$ as an $\OO_{\PP^{g-3}}$-module by an Eagon-Northcott type complex $\cC^b\otimes \OO_{\PP^{g-3}}(a)$ where the $j^{th}$ term of $\cC^b_j$ is defined as
$$
\cC^b_j=\begin{cases}\bigwedge^j F\otimes S_{b-j}G\otimes \OO_\PPr(-j)\ &{\text{for} \ 0 \leq j\leq b}\\ 
\bigwedge^{j+1}F\otimes D_{j-b-1}G^*\otimes  \OO_\PPr(-j-1) &{\text{for}\ j\geq b+1\ \ }\end{cases}
$$
and whose differentials $\delta_j:\cC^b_j\longrightarrow\cC^b_{j-1}$ are given by the multiplication by
\begin{align*}
&\Phi\in H^0(\PPE, F^*\otimes G) & \text{for } j\neq b+1 \\ 
&\bigwedge^2\Phi \in H^0(\PPE, \wedge^2 F^*\otimes \wedge^2 G) & \text{for }  j=b+1
\end{align*}
(see \cite{BE75}).  
The resolution in (\ref{relRes}) induces a double complex of Eagon--Northcott type resolutions 
\begin{align}\label{doubleComplex}
 \cC^0 \leftarrow \cC^{g-7}\otimes \OO_{\PP^{g-3}}(-3).
\end{align}

\subsection{Mapping cone construction and examples}\label{mappingCone}

Since all maps in the double complex in (\ref{doubleComplex}) are minimal, the mapping cone $[\cC^0 \leftarrow \cC^{g-7}\otimes \OO_{\PP^{g-3}}(-3)]$ of this double complex induces a minimal free resolution of $C\subset \PP^{g-3}$ (see \cite{Eis}). 

\begin{example}
 For a trigonal curve of genus $g=7$ the double complex is
$$
\begin{xy}
 \xymatrix{
 0 & \ar[l] \OO_C & \ar[l] \OO_\PPE & \ar[l] \OO_\PPE(-3H) & \ar[l] 0 \\
  & & \ar[u] \OO_{\PP^4} & \ar[u] \ar[l] \OO_{\PP^4}(-3) & \\  
  & & \ar[u] \bigwedge^2 F\otimes \OO_{\PP^4}(-2) & \ar[u] \ar[l] \bigwedge^2 F\otimes \OO_{\PP^4}(-5) & \\
  & & \ar[u]  \bigwedge^3 F\otimes DG^* \otimes \OO_{\PP^{4}}(-3) & \ar[u] \ar[l]  \bigwedge^3 F\otimes DG^* \otimes \OO_{\PP^{4}}(-6) & \\
  & & \ar[u] 0 &  \ar[u] 0 &
 }
\end{xy}
$$
In the mapping cone construction we sum up the terms on the diagonal in the double complex (with the right differentials). This yields 
$$
0\leftarrow \OO_C \leftarrow \OO_{\PP^4} \longleftarrow 
\begin{matrix} \bigwedge^2 F\otimes \OO_{\PP^4}(-2) \\ \oplus \\ \OO_{\PP^4}(-3) \end{matrix} \longleftarrow 
\begin{matrix} \bigwedge^3 F\otimes DG^* \otimes \OO_{\PP^{4}}(-3) \\ \oplus \\ \bigwedge^2 F\otimes \OO_{\PP^4}(-5)\end{matrix} 
\leftarrow \bigwedge^3 F\otimes DG^* \otimes \OO_{\PP^{4}}(-6) \leftarrow 0,
$$
 and the Betti table of the minimal free resolution of $C\subset \PP^4$ is 
$$
\begin{tabular}{c|ccccc}
 \multicolumn{1}{@{}}{\ }&& \\[-2mm] 
  & $0$ & $1$ & $2$ &  $3$   \\ \hline 
$0$ & 1 & 0 & 0 & 0 \\
$1$ & 0 & 3 & 2 & 0 \\
$2$ & 0 & 1 & 0 & 0 \\
$3$ & 0 & 0 & 3 & 2 
 \\[-2mm]
 \multicolumn{1}{@{}}{\ }&&
\end{tabular}
$$
\end{example}

For $g\ge 8$ the construction of Section \ref{resOnScroll} yields the following double complex 
$$
\begin{xy}
 \xymatrix{
 0 & \ar[l] \OO_C & \ar[l] \OO_\PPE & \ar[l] \OO_\PPE(-3H + (g-7)R) & \ar[l] 0 \\
  & & \ar[u] \OO_{\PP^{g-3}} & \ar[u] \ar[l] S_{g-7}G\otimes \OO_{\PP^{g-3}}(-3) & \\  
  & & \ar[u] \bigwedge^2 F\otimes \OO_{\PP^{g-3}}(-2) & \ar[u] \ar[l] F\otimes S_{g-8}G\otimes \OO_{\PP^{g-3}}(-4) & \\
  & & \ar[u]  \bigwedge^3 F\otimes DG^* \otimes \OO_{\PP^{g-3}}(-3) & \ar[u] \vdots & \\
  & & \ar[u] \vdots & &
 }
\end{xy}
$$
and the mapping cone induces the minimal free resolution  
$$
0\leftarrow \OO_C \leftarrow \OO_{\PP^{g-3}} \longleftarrow 
\begin{matrix} \bigwedge^2 F\otimes \OO_{\PP^{g-3}}(-2) \\ \oplus \\ S_{g-7}G\otimes \OO_{\PP^{g-3}}(-3) \end{matrix} \longleftarrow 
\begin{matrix} \bigwedge^3 F\otimes DG^* \otimes \OO_{\PP^{g-3}}(-3) \\ \oplus \\ F\otimes S_{g-8}G\otimes \OO_{\PP^{g-3}}(-5)\end{matrix} 
\longleftarrow \dots
$$
The shape of the minimal free resolution of $C$ as an $\OO_{\PP^{g-3}}$-module is given as stated in the main theorem. Indeed, we note that the second last map in $\cC^{g-7}\otimes \OO_{\PP^{g-3}}(-3)$ is of degree $2$ and occurs in the $(g-6)^{th}$ step in the mapping cone construction. The theorem follows.

\begin{example}
We end this section with a trigonal curve $C$ of genus $6$. The minimal free resolution of $C\subset \PP^3 = \PP(H^0(C,K_C\otimes T^{-1})^*)$ has the following shape 
$$
\begin{tabular}{c|cccc}
 \multicolumn{1}{@{}}{\ }&& \\[-2mm] 
  & $0$ & $1$ & $2$    \\ \hline 
$0$ & 1 & 0 & 0  \\
$1$ & 0 & 1 & 0  \\
$2$ & 0 & 0 & 0  \\
$3$ & 0 & 2 & 0  \\
$4$ & 0 & 0 & 1
 \\[-2mm]
 \multicolumn{1}{@{}}{\ }&&
\end{tabular}
$$
We can still apply the above method to compute a minimal free resolution. Note that $C$ has quartic generators induced by $\cC^{-1}\otimes \OO_{\PP^3}(-3)$.
\end{example}

\section{Higher residuals of $T$ with respect to the canonical bundle}\label{higherResiduals}

In this section we extend Theorem \ref{mainThm} to trigonal curves of genus $g$ embedded by line bundles of the form $K_C\otimes T^{-n}$ for $n\ge 1$ and $g\ge 3n+4$.  
We fix the notation of this section. For a trigonal curve $C$ of genus $g\ge 5$ let $T$ be the unique trigonal bundle and let $m$ be the Maroni-invariant (and $a=g-2-m$). 

\subsection{Very ampleness of $K_C\otimes T^{-n}$}\label{veryAmpleness}

We have the following two lemmata. 

\begin{lemma}
 For $n > m$ the line bundle $K_C\otimes T^{-n}$ does not separate points of the morphism $C\stackrel{3:1}{\longrightarrow} \PP^1$ induced by $T$. In particular, $K_C\otimes T^{-n}$ is not very ample. 
\end{lemma}

\begin{proof} 
We may assume that $n=m+1$ since $H^0(C,K_C\otimes T^{-n_1})\subset H^0(C,K_C\otimes T^{-n_2})$ for $n_1\ge n_2$. 
Let $D = p+q+r \in |T|$ be a divisor. There is a short exact sequence 
$$
0 \to T^{m+1} \to T^{m+1}(D) \to T^{m+1}(D)|_D \to 0.
$$
The long exact sequence 
$$
0 \to H^0(C,T^{m+1}) \to H^0(C,T^{m+2}) \to \underbrace{\Gamma(T^{m+1}(D)|_D)}_{3-\text{dimensional}}\to H^1(C,T^{m+1}) \to H^1(C,T^{m+2}) \to 0
$$
is induced by the global section functor. Since $m+2=\min\{n\in \NN\ :\ h^0(C,T^n)>n+1\}$, the difference is 
$$
h^0(C,T^{m+2})-h^0(C,T^{m+1})\ge 2.
$$
Hence, by the long exact sequence $h^1(C,T^{m+1})-h^1(C,T^{m+2})\le 1$ and for $p,q \in \Supp(D)$ 
\begin{align*}
 h^0(C,K_C\otimes T^{-m-1})-h^0(C,K_C\otimes T^{-m-1}(-p-q)) 
 & \le  h^0(C,K_C\otimes T^{-m-1})-h^0(C,K_C\otimes T^{-m-2}) \\
 & =  h^1(C,T^{m+1})-h^1(C,T^{m+2}) \\ 
 & \le  1
\end{align*}
Thus, we can not find a section separating the point $p$ and $q$. The lemma follows.
\end{proof}

\begin{lemma}{\cite[V, Theorem 2.17]{Ha} or \cite[Lemma 1, p. 176]{MS86}}\label{higherResVeryAmple}
 For $n < m$ the line bundle $K_C\otimes T^{-n}$ on $C$ is very ample. The line bundle $K_C\otimes T^{-m}$ is generated by global sections and separates points of the morphism $C\stackrel{3:1}{\longrightarrow} \PP^1$ induced by $T$.
\end{lemma}

\begin{remark}
 For $g \ge 3m + 3$ the line bundle $K_C\otimes T^{-m}$ is very ample since $h^0(C,T^m(p+q)) = h^0(C,T^m) = m+1$. Indeed, by \cite{Mar46} or \cite[Proposition 1]{MS86} the Brill--Noether locus $W^{m+1}_{3m+2}(C)$ is empty for $g \ge 3m + 3$ and that implies $h^0(C,T^m(p+q)) = m+1$. This can only happen in few cases since $g\le 3m + 4$ always holds by the inequalities (\ref{MaroniInequalities}). 
 Furthermore, the bound $g\ge 3m+3$ is satisfied for our assumptions in the next section.
\end{remark}

\subsection{The minimal free resolution of $C\subset \PP(H^0(C,K_C\otimes T^{-n})^*)$}

We fix an integer $n\ge 1$. In order to ensure that the linear system $K_C\otimes T^{-n}$ on a curve $C$ of genus $g$ is very ample, we have to assume that $n$ is less than or equal to the minimal possible Maroni-invariant of a curve of genus $g$. By (\ref{MaroniInequalities}), this yields the bound $n\le \frac{g-4}{3}$, or equivalently 
$$
g\ge 3n+4.
$$

We use the same strategy as in the previous section to show that the curve $C$ is projectively normal in $\PP(H^0(C,K_C\otimes T^{-n})^*)$. We have the following generalisation of Theorem \ref{mainThm}.

\begin{theorem}\label{genMainThm}
 Let $n\ge 1$ be an integer. Let $C$ be a trigonal curve of genus $g\ge 3n+4$ and let $T$ be the unique trigonal bundle. The Betti table of a minimal free resolution of $C\subset \PP(H^0(C,K_C\otimes T^{-n})^*)=\PP^{g-2n-1}$ has the following shape 
$$
\begin{tabular}{c|cccccccc}
 \multicolumn{1}{@{}}{\ }&& \\[-2mm] 
    & $0$ & $1$ & $\dots$ & $g-3n-3$ & $g-3n-2$ & $\dots$ & $g-2n-3$ & $g-2n-2$  \\ \hline 
  $0$ & 1 & 0 & $\cdots$ & 0 & 0 & $\cdots$ & 0 & 0 \\
  $1$ & 0 & * & $\cdots$ & * & * & $\cdots$ & * & 0 \\
  $2$ & 0 & * & $\cdots$ & * & 0 & $\cdots$ & 0 & 0 \\
  $3$ & 0 & 0 & $\cdots$ & 0 & * & $\cdots$ & * & *
  \\[-2mm]
 \multicolumn{1}{@{}}{\ }&&
\end{tabular}
$$
 where $*$ indicates nonzero entries. In particular, the embedded curve is projectively normal.  
\end{theorem}

The nonzero Betti numbers are given as in \cite[Theorem 2.4]{Na99} for $c = g-2n-2$ and $p = g-3n-3$.

\begin{proof}
By Section \ref{veryAmpleness}, let $C\subset \PP(H^0(C,K_C\otimes T^{-n})^*)=:\PP^{g-2n-1}$ be a smooth embedded curve of genus $g$ and of degree $\deg(C) = 2g - 3n - 2$. Note that 
$g - 2n - 1 \ge 
n+3 \ge 4$.
Let 
$$
X= \bigcup_{D\in |T|} \overline{D} \subset \PP^{g-2n-1}
$$
be the rational normal scroll swept out by the trigonal bundle. Since $n+1<m+2$, 
$$
h^0(C, K_C\otimes T^{-n}) - h^0(C, K_C\otimes T^{-n-1}) =2 
$$
by (\ref{MaroniAsMin}). Hence, $X$ is a rational normal surface of degree $g-2n-2$. Let $\cE = \OO_{\PP^1}(g-2-m-n)\oplus \OO_{\PP^1}(m-n)$ be a rank $2$ bundle. Then the rational normal surface $X$ is the image of $\PPE$ as in Section 
\ref{resOnScroll}. Since $T$ is base point free, we consider $C\subset \PPE$. We denote again $H$ and $R$ the generators of the Picard group of $\PPE$ as in Section \ref{resOnScroll} such that $H^2 = g-2n-2$, $H.R= 1$ and $R^2 = 0$ as well as $\OO_\PPE(H)\otimes \OO_C = K_C\otimes T^{-n}$ and $\OO_\PPE(R)\otimes \OO_C = T$. 
 \\
The resolution of $\OO_C$ as an $\OO_\PPE$-module is
\begin{align*}
 \begin{xy}
  \xymatrix{
  0 & \ar[l] \OO_C & \ar[l] \OO_\PPE & \ar[l] \OO_\PPE(-aH + bR) & \ar[l] 0 \\
  }
 \end{xy}
\end{align*} 
for integers $a,b$ since $C$ is a divisor on $\PPE$. We have 
$$
(aH-bR).R = 3 \text{ and } (aH - bR).H = \deg(C) = 2g -3n-2.
$$
Hence, $a=3$ and $b = g-3n-4\ge 0$ and thus the $\OO_\PPE$-modules in 
\begin{align*}
 \begin{xy}
  \xymatrix{
  0 & \ar[l] \OO_C & \ar[l] \OO_\PPE & \ar[l] \OO_\PPE(-3H + (g-3n-4)R) & \ar[l] 0 \\
  }
 \end{xy}
\end{align*} 
can be resolved as $\OO_{\PP^{g-2n-1}}$-modules by Eagon--Northcott type resolutions (see Section \ref{resOnScroll}). As in Section \ref{mappingCone}, the mapping cone of the induced double complex 
$$
[\cC^0 \leftarrow \cC^{g-3n-4}\otimes \OO_{\PP^{g-2n-1}}(-3)]
$$
is a minimal free resolution of $C\subset \PP^{g-2n-1}$. The shape of its Betti table is given as in Theorem \ref{genMainThm}. Indeed, the length of $\cC^0$ is $g-2n-3$ and the degree $2$ syzygies in $\cC^{g-3n-4}$ are in the step $j = b + 1 = g-3n-3$. Since $\cC^{g-3n-4}$ is homologically shifted by $1$ in the mapping cone construction, these syzygies appear in step $g-3n-2$ as stated in the theorem. 
\end{proof}


\end{document}